\documentclass[11pt]{amsart}
\usepackage{amsmath,amsthm,amssymb,enumitem}
\usepackage{lipsum,mathtools,tikz}
\usetikzlibrary{shapes.geometric}

\theoremstyle{plain}
\newtheorem{mainthm}{Theorem}
\newtheorem{thm}{Theorem}[section]
\newtheorem{lem}[thm]{Lemma}
\newtheorem{prop}[thm]{Proposition}

\theoremstyle{definition}
\newtheorem{defn}[thm]{Definition}
\newtheorem{example}[thm]{Example}

\theoremstyle{remark}
\newtheorem{rem}[thm]{Remark}

\newcommand{\ncp}{\textsc{NC}}
\newcommand{\pkfn}{\textsc{PF}}
\newcommand{\pos}{\textsc{Poset}}
\newcommand{\st}{\textsc{Star}}

\newcommand{\link}{\textsc{Link}}
\newcommand{\bool}{\textsc{Bool}}

\newcommand{\bv}{\textbf{v}}
\newcommand{\bu}{\textbf{u}}

\tikzstyle{BlueLine}=[line width=0.3mm,color=blue,text=black]
\tikzstyle{BluePoly}=[BlueLine,fill=blue!20]
\tikzstyle{RedLine}=[line width=0.3mm,color=red,text=black]
\tikzstyle{RedPoly}=[RedLine,fill=red!20]
\tikzstyle{GreenLine}=[thick,color=black!30!green,text=black]
\tikzstyle{OrangeLine}=[thick,color=orange]
\tikzstyle{GrayLine}=[thick,color=black!50!gray]
\tikzstyle{GrayPoly}=[GrayLine,fill=gray!20]
\tikzstyle{dot}=[shape=circle,draw,color=black,fill=black,inner sep=1.5pt]
\tikzstyle{littledot}=[dot,inner sep=0.75pt]
\tikzstyle{disk}=[thick,shape=circle,draw,color=black,fill=yellow!10]
\tikzstyle{plate}=[thick,shape=rectangle,draw,color=black,fill=yellow!10,rounded corners,minimum size=1.1cm]
\newcommand{\makepoints}{\foreach \n in {1,...,4} {\coordinate (\n) at (\n*-90-135:0.5cm);};}
\newcommand{\drawpoints}{\foreach \n in {1,...,4} {\draw (\n) node [littledot] {};};}

\begin{document}

\title[Undesired parking spaces]{Undesired parking spaces and contractible pieces of the
  noncrossing partition link} 

\author{Michael Dougherty} \email{dougherty@math.ucsb.edu}
\address{Department of Mathematics, UC Santa Barbara, Santa Barbara,
  CA 93106} 

\author{Jon McCammond} \email{jon.mccammond@math.ucsb.edu}
\address{Department of Mathematics, UC Santa Barbara, Santa Barbara,
  CA 93106} 

\date{\today}

\begin{abstract}
  There are two natural simplicial complexes associated to the
  noncrossing partition lattice: the order complex of the full lattice
  and the order complex of the lattice with its bounding elements
  removed.  The latter is a complex that we call the \emph{noncrossing
    partition link} because it is the link of an edge in the former.
  The first author and his coauthors conjectured that various
  collections of simplices of the noncrossing partition link
  (determined by the undesired parking spaces in the corresponding
  parking functions) form contractible subcomplexes.  In this article
  we prove their conjecture by combining the fact that the star of a
  simplex in a flag complex is contractible with the second author's
  theory of noncrossing hypertrees. 
\end{abstract}

\maketitle

The \emph{noncrossing partition lattice} $\ncp_n$ is a fundamental
object in modern combinatorics and from it one can construct two
natural simplicial complexes: the order complex of the full lattice
and the order complex of the lattice with its bounding elements
removed.  The former has a natural piecewise euclidean metric that has
been the subject of some study, particularly for its connections to
curvature of the $n$-strand braid group \cite{bradymccammond10, hks16,
  dmw-cat0}.  The latter has a natural piecewise spherical metric and
we call it the \emph{noncrossing partition link} $\link(\ncp_n)$
because it is the link of an edge in the former
\cite{bradymccammond10}.  Since the top-dimensional simplices in the
noncrossing partition link are in natural bijection with the maximal
chains in $\ncp_n$, the factorizations of the $n$-cycle
$(1,2,\ldots,n)$ into $n-1$ transpositions, the parking functions of
length $n-1$ and the properly ordered noncrossing trees on $n$ labeled
vertices, any of these sets can be used as labels on these simplices
\cite{stanley97, mccammond-ncht}. The first result we prove is the
following.

\begin{mainthm}[Unused boundary edge]\label{main:edge}
  The simplices of the noncrossing partition link labeled by properly
  ordered noncrossing trees that omit a fixed boundary edge $e$ form a
  contractible subcomplex.
\end{mainthm}

Our proof combines an easy result about flag complexes with the second
author's theory of noncrossing hypertrees \cite{mccammond-ncht}.
Theorem~\ref{main:edge} can also be stated in terms of parking
functions.

\begin{mainthm}[Undesired last parking space]\label{main:last}
  The simplices of the noncrossing partition link labeled by parking
  functions where no car wants to park in the last parking space form
  a contractible subcomplex.
\end{mainthm}

In this form, Theorem~\ref{main:last} is one part of the conjecture
that the first author and his coauthors made in \cite{reu16}.  The
more general conjecture involves filtering the collection of parking
functions by the last parking space that is undesired.  We also prove
the general version.

\begin{mainthm}[Undesired parking space]\label{main:space}
  The simplices of the noncrossing partition link labeled by the
  collection of parking functions for which $k$ is the number of the
  largest undesired parking space form a contractible subcomplex.
\end{mainthm}
 
Theorem~\ref{main:space} quickly follows from Theorem~\ref{main:last}
and a few elementary lemmas.  These results have also been proved by
Henri M\"uhle using a completely different approach, namely, by
embedding the relevant poset into a supersolvable lattice.  See
\cite{muehle} for details.

The structure of the article is as follows.  After proving that the
star of a simplex in a flag complex is contractible, we review the
theory of noncrossing hypertrees and prove Theorem~\ref{main:edge}.
When re-interpreted in terms of parking functions
Theorem~\ref{main:edge} becomes Theorem~\ref{main:last} and then
Theorem~\ref{main:space} is derived as an easy corollary.

\section{Simplices}\label{sec:simplices}

This section recalls an easy way to deform a simplex onto a subspace.

\begin{defn}[Simplices]\label{def:simplex}
  A \emph{simplex} $\sigma$ is the convex hull of a finite set $S$ of
  points in general position in some euclidean space and the convex
  hull of any proper subset of $S$ is a \emph{proper subsimplex of
    $\sigma$}.  For any simplex $\sigma$ and proper subsimplex $\tau$
  let $\sigma \setminus \tau$ denote the subspace of $\sigma$ that is
  the union of all the subsimplices of $\sigma$ that do not contain
  all of the vertices of $\tau$.  If $\bv_1,\ldots,\bv_n$ are the
  vertices of $\sigma$,  then each point $\mathbf{p}$ in
  $\sigma$ has a unique description $\mathbf{p} = a_1\bv_1 + \cdots
  +a_n\bv_n$ with each $a_i \geq 0$ and $\sum_i a_i = 1$.  These
  unique scalars $a_i$ are the \emph{barycentric coordinates} of
  $\mathbf{p}$.  For each point $\mathbf{p}$, the vertices $\bv_i$
  corresponding to the non-zero $a_i$ coordinates determine the unique
  smallest subsimplex of $\sigma$ that contains $\mathbf{p}$.  The
  points in $\sigma \setminus \tau$ are precisely those where at least
  one of the vertices in $\tau$ has $0$ as its barycentric coordinate.
  The \emph{barycenter} of a subsimplex $\tau$ is the average of its
  vertices.  Concretely, if $\tau$ has vertices $\bv_1, \ldots,
  \bv_k$, then $\mathbf{p} = \frac{1}{k}(\bv_1 + \cdots + \bv_k)$ is
  its barycenter.
 \end{defn}

\begin{prop}[Removing a proper face]\label{prop:deform}
  Let $\sigma$ be a simplex.  For each proper subsimplex $\tau \subset
  \sigma$ there is a deformation retraction from $\sigma$ to
  $\sigma\setminus\tau$.
\end{prop}

\begin{proof}
  Since $\tau$ is a proper subsimplex, there are vertices of $\sigma$ not in
  $\tau$ and the collection of all such vertices spans a disjoint
  subsimplex that we call $\tau'$.  Let $\bv_{\tau,\tau'}$ be the
  vector from the barycenter of $\tau$ to the barycenter of $\tau'$.
  For each point $\mathbf{p}$ in $\sigma$ the line $\mathbf{p} +
  t\cdot \bv_{\tau,\tau'}$ through $\mathbf{p}$ in the
  $\bv_{\tau,\tau'}$ direction intersects the subcomplex $\sigma
  \setminus \tau$ in a single point.  The function $f:\sigma \to
  \sigma$ that sends each point in $\sigma$ to this well-defined point
  in $\sigma \setminus \tau$ is a continuous retraction and, since
  $\sigma$ is convex, there is a straight-line homotopy between the
  identity map and $f$.  This is a deformation retraction from
  $\sigma$ to $\sigma \setminus \tau$ which leaves the points in the
  subcomplex fixed throughout.
\end{proof}

\begin{rem}[Removing a proper face]\label{rem:deform}
  The retraction $f$ is easy to describe in barycentric coordinates.
  Let $\bv_1, \bv_2, \ldots, \bv_k$ be the vertices of $\tau$, let
  $\bu_1, \bu_2, \ldots, \bu_\ell$ be the vertices of $\tau'$ and let
  \[
  \mathbf{p} = a_1 \bv_1 + \cdots + a_k \bv_k + b_1 \bu_1 + \cdots +
  b_\ell \bu_\ell
  \] 
  (with all $a_i, b_j \geq 0$ and $\sum_i a_i +
  \sum_j b_j = 1$) be an arbitrary point in $\sigma$.  Then
  $\bv_{\tau,\tau'} = -\frac{1}{k}(\bv_1 + \bv_2 + \ldots + \bv_k) +
  \frac{1}{\ell}(\bu_1 + \bu_2 + \ldots + \bu_\ell)$ and the unique
  point on the line $\mathbf{p} + t \cdot \bv_{\tau,\tau'}$ that lies
  in $\sigma \setminus \tau$ occurs when $t = m \cdot k$, where $m$ is
  the minimum value of the set $\{a_1,a_2,\ldots,a_k\}$ and $k$ is the
  number of vertices in $\tau$.  For smaller values of $t$, all of the
  $a_i$ coordinates remain positive and the point is not yet in
  $\sigma \setminus \tau$.  For larger values of $t$, at least one of
  the $a_i$ coordinates is negative and the point lies outside the
  simplex $\sigma$.
\end{rem}

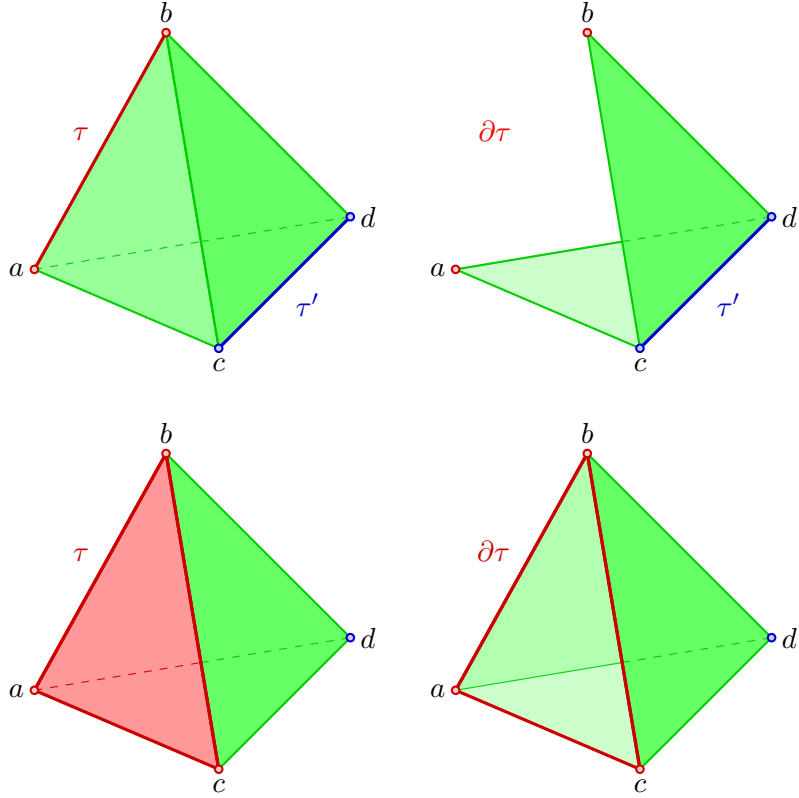
\begin{figure}
  \begin{tikzpicture}[scale=.7]
    \tikzstyle{redline}=[very thick,color=red!80!black]
    \tikzstyle{greenline}=[thick,color=green!80!black]
    \tikzstyle{blueline}=[very thick,color=blue!80!black]
    \tikzstyle{bluedot}=[shape=circle,draw,thick,color=blue!80!black,fill=blue!20,inner sep=1pt]
    \tikzstyle{reddot}=[shape=circle,draw,thick,color=red!80!black,fill=red!20,inner sep=1pt]
    \tikzstyle{redpoly}=[redline,fill=red!40,line join=round]
    \tikzstyle{greenpoly}=[thick,color=green!80!black,fill=green!60]
    \begin{scope}[shift={(-4,4)}]
      \coordinate (p1) at (-3,-1.5) {};
      \coordinate (p2) at (-.5,3) {};
      \coordinate (p3) at (.5,-3) {};
      \coordinate (p4) at (3,-.5) {};
      \coordinate (p5) at (intersection of p1--p4 and p2--p3);
      \node[left] at (p1) {$a$};
      \node[above] at (p2) {$b$};
      \node[below] at (p3) {$c$};
      \node[right] at (p4) {$d$};
      \filldraw[greenpoly] (p2)--(p4)--(p3);
      \filldraw[greenpoly,fill=green!40] (p1)--(p2)--(p3)--cycle;
      \draw[dashed,green!80!black] (p1)--(p5);
      \draw[dashed,green!80!black] (p5)--(p4);
      \draw[greenline] (p2)--(p3);
      \draw[redline] (p1) -- node[above left] {$\tau$} (p2);
      \draw[blueline] (p3) -- node[below right] {$\tau'$} (p4);
      \node[reddot] at (p1) {};
      \node[reddot] at (p2) {};
      \node[bluedot] at (p3) {};
      \node[bluedot] at (p4) {};
    \end{scope}
        
    \begin{scope}[shift={(4,4)}]
      \coordinate (p1) at (-3,-1.5) {};
      \coordinate (p2) at (-.5,3) {};
      \coordinate (p3) at (.5,-3) {};
      \coordinate (p4) at (3,-.5) {};
      \coordinate (p5) at (intersection of p1--p4 and p2--p3);
      \node[left] at (p1) {$a$};
      \node[above] at (p2) {$b$};
      \node[below] at (p3) {$c$};
      \node[right] at (p4) {$d$};
      \filldraw[color=green!20] (p1)--(p5)--(p3)--cycle;
      \filldraw[greenpoly] (p2)--(p4)--(p3)--cycle;
      \draw[color=white,text=red] (p1) -- node[above left] {$\partial\tau$} (p2);
      \draw[blueline] (p3) -- node[below right] {$\tau'$} (p4);
      \draw[greenline] (p3)--(p1)--(p5);
      \draw[dashed,green!80!black] (p5)--(p4);
      \node[reddot] at (p1) {};
      \node[reddot] at (p2) {};
      \node[bluedot] at (p3) {};
      \node[bluedot] at (p4) {};
    \end{scope}

    \begin{scope}[shift={(-4,-4)}]
      \coordinate (p1) at (-3,-1.5) {};
      \coordinate (p2) at (-.5,3) {};
      \coordinate (p3) at (.5,-3) {};
      \coordinate (p4) at (3,-.5) {};
      \coordinate (p5) at (intersection of p1--p4 and p2--p3);
      \node[left] at (p1) {$a$};
      \node[above] at (p2) {$b$};
      \node[below] at (p3) {$c$};
      \node[right] at (p4) {$d$};
      \filldraw[greenpoly] (p2)--(p4)--(p3);
      \filldraw[redpoly] (p1)--(p2)--(p3)--cycle;
      \draw[dashed,red!80!black] (p1)--(p5);
      \draw[dashed,green!80!black] (p5)--(p4);
      \draw[redline] (p2)--(p3);
      \draw[redline] (p1) -- node[above left] {$\tau$} (p2);
      \node[reddot] at (p1) {};
      \node[reddot] at (p2) {};
      \node[reddot] at (p3) {};
      \node[bluedot] at (p4) {};
    \end{scope}
        
    \begin{scope}[shift={(4,-4)}]
      \coordinate (p1) at (-3,-1.5) {};
      \coordinate (p2) at (-.5,3) {};
      \coordinate (p3) at (.5,-3) {};
      \coordinate (p4) at (3,-.5) {};
      \coordinate (p5) at (intersection of p1--p4 and p2--p3);
      \node[left] at (p1) {$a$};
      \node[above] at (p2) {$b$};
      \node[below] at (p3) {$c$};
      \node[right] at (p4) {$d$};
      \filldraw[color=green!30] (p1)--(p2)--(p5)--cycle;
      \filldraw[color=green!20] (p1)--(p5)--(p3)--cycle;
      \filldraw[greenpoly] (p2)--(p4)--(p3)--cycle;
      \draw[redline] (p1)--(p2)--(p3)--cycle;
      \draw[green!80!black] (p1)--(p5);
      \draw[dashed,green!80!black] (p5)--(p4);
      \draw[redline] (p2)--(p3);
      \draw[redline] (p1) -- node[above left] {$\partial\tau$} (p2);
      \node[reddot] at (p1) {};
      \node[reddot] at (p2) {};
      \node[reddot] at (p3) {};
      \node[bluedot] at (p4) {};
    \end{scope}
  \end{tikzpicture}
  \caption{Let $\sigma$ be the tetrahedron $abcd$.  The top row shows
    the deformation retraction from $\sigma$ to $\sigma \setminus
    \tau$ when $\tau$ is the edge $ab$.  The bottom row shows the
    deformation retraction when $\tau$ is the triangle
    $abc$.\label{fig:deform}}
\end{figure}

\begin{example}[Removing a proper face]\label{ex:deform}
  Two of these deformations are shown in Figure~\ref{fig:deform}.  In
  the top row $\tau$ is the edge $ab$, $\tau'$ is the edge $cd$ and
  the direction of deformation is the vector from the midpoint of
  $\tau$ to the midpoint of $\tau'$.  In the bottom row, $\tau$ is the
  triangle $abc$, $\tau'$ is the point $d$ and the direction of
  deformation is the vector from the center of $\tau$ to the
  point~$d$.
\end{example}

\section{Flag Complexes}\label{sec:flag}

In this section we prove that the star of a simplex in a flag complex
is contractible (Proposition~\ref{prop:stars-contract}).  This is
surely a well-known fact but since we have been unable to locate a
reference in the literature, we provide an easy proof using the
deformation retractions described in Section~\ref{sec:simplices}.

\begin{defn}[Stars and links]\label{def:star-link}
  Let $\rho$ be a simplex in a simplicial complex $X$. The
  \emph{star of $\rho$} is the smallest subcomplex $\st(\rho)$
  in $X$ which contains every simplex that has a nontrivial
  intersection with $\rho$. We
  distinguish three types of simplices in $\st(\rho)$ depending on
  whether it is contained in $\rho$, disjoint from $\rho$, or neither.
  A simplex $\tau$ disjoint from $\rho$ is a \emph{link simplex} and
  the set of all link simplices form a simplicial complex
  $\link(\rho)$ called the \emph{link of $\rho$}.  A simplex $\sigma$
  that is neither in the link nor contained in $\rho$ is called a
  \emph{connecting simplex}. Note that by the minimality condition in
  the definition of $\st(\rho)$, every link simplex $\tau$ is a
  proper subsimplex of some connecting simplex $\sigma$.
\end{defn}

Recall that a simplical complex $X$ is a \emph{flag complex} when each
complete graph in the $1$-skeleton of $X$ is the $1$-skeleton of a
simplex in $X$.

\begin{lem}[Maximal connecting simplices]\label{lem:max-connect}
  Let $\rho$ be a simplex in a flag complex $X$.  For every simplex
  $\tau \in \link(\rho)$ there is a unique maximal simplex $\sigma \in
  \st(\rho)$ such that $\sigma \cap \link(\rho) = \tau$.  In
  particular, $\sigma' \in \st(\rho)$ and $\sigma' \cap \link(\rho) =
  \tau$ implies $\sigma' \subset \sigma$ for this maximal $\sigma$.
  Moreover, $\sigma$ is a connecting simplex that properly contains
  $\tau$.
\end{lem}

\begin{proof}
  Let $V$ be the set of vertices in $\tau$ together with any vertex of
  $\rho$ that is connected to every vertex of $\tau$ by an edge.  Note
  that the restriction of the $1$-skeleton of $X$ to the vertex set
  $V$ is a complete graph by the definition of $V$ and by the fact
  that both $\tau$ and $\rho$ are simplices.  Because $X$ is a flag
  complex, there is a simplex $\sigma$ with $V$ as its vertex set and
  its maximality properties are immediate.  Finally, by the remark in
  Definition~\ref{def:star-link}, there is some connecting simplex
  $\sigma'$ that contains $\tau$ and by replacing $\sigma'$ by a
  subsimplex if necessary we may also assume that the vertices of
  $\sigma'$ are contained in the set $V$.  Thus at least one vertex in
  $\rho$ is connected to every vertex of $\tau$, which means that
  $\sigma$ is a connecting simplex properly containing $\tau$.
\end{proof}

\begin{defn}[Filtrations]\label{def:filter}
  Let $\rho$ be a simplex in a simplicial complex $X$, let $L =
  \link(\rho)$ and let $S = \st(\rho)$.  For each integer $k$ let
  $L^{(k)}$ be the $k$-skeleton of $L$ and let $S^{(k)}$ be the
  largest subcomplex of $S$ such that $S^{(k)} \cap L =
  L^{(k)}$.  Note that when $k=-1$, $L^{(-1)} = \emptyset$ and
  $S^{(-1)} = \rho$.  These subcomplexes allow us to write $L$ and $S$
  as nested unions $L = \cup_k L^{(k)}$ and $S = \cup_k S^{(k)}$ which
  we call the \emph{natural filtrations} of $L$ and $S$.  We should
  note that when $1 \leq k < \dim(L)$, neither $L^{(k)}$ nor
  $S^{(k)}$ is a flag complex.
\end{defn}

\begin{example}[Filtrations]\label{ex:filter}
  Let $\rho$ be a triangle in the standard triangular tiling of the
  plane.  The natural filtration of the star of $\rho$ is shown in
  Figure~\ref{fig:filter}.  The subcomplex $S^{(1)} = \st(\rho)$ is shown
  on the left, the subcomplex $S^{(0)}$ is shown in the middle and the
  subcomplex $S^{(-1)} = \rho$ is shown on the right.
\end{example}

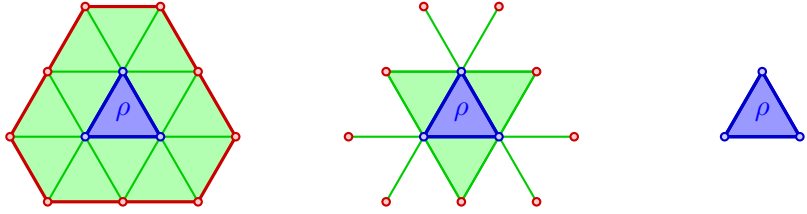
\begin{figure}
  \begin{tikzpicture}
    \tikzstyle{redline}=[very thick,color=red!80!black]
    \tikzstyle{greenline}=[thick,color=green!80!black]
    \tikzstyle{blueline}=[very thick,color=blue!80!black]
    \tikzstyle{bluedot}=[shape=circle,draw,thick,color=blue!80!black,fill=blue!20,inner sep=1pt]
    \tikzstyle{reddot}=[shape=circle,draw,thick,color=red!80!black,fill=red!20,inner sep=1pt]
    \tikzstyle{bluepoly}=[blueline,fill=blue!40,line join=round]
    \tikzstyle{greenpoly}=[thick,color=green!80!black,fill=green!30]
    \begin{scope}[shift={(-4.5,0)}]
      \coordinate (0) at (-0.5,{sqrt(3)/2});
      \coordinate (1) at (0.5,{sqrt(3)/2});
      \coordinate (2) at (0,{sqrt(3)});
      \coordinate (3) at (-1,0);
      \coordinate (4) at (0,0);
      \coordinate (5) at (1,0);
      \coordinate (6) at (1.5,{sqrt(3)/2});
      \coordinate (7) at (1,{sqrt(3)});
      \coordinate (8) at (0.5,{sqrt(3)*3/2});
      \coordinate (9) at (-0.5,{sqrt(3)*3/2});
      \coordinate (10) at (-1,{sqrt(3)});
      \coordinate (11) at (-1.5,{sqrt(3)/2});
      \filldraw[greenpoly] (3)--(4)--(5)--(6)--(7)--(8)--(9)--(10)--(11)--cycle;
      \filldraw[bluepoly] (0)--(1)--(2)--cycle;
      \node () at (0,1.2) {$\textcolor{blue}{\rho}$};
      \draw[greenline] (3)--(8) (4)--(7) (5)--(9) (4)--(10) (6)--(11) (7)--(10);
      \draw[blueline] (0)--(1)--(2)--cycle;
      \draw[redline] (3)--(4)--(5)--(6)--(7)--(8)--(9)--(10)--(11)--cycle;
      \foreach \n in {0,...,2} {\draw (\n) node [bluedot] {};}
      \foreach \n in {3,...,11} {\draw (\n) node [reddot] {};}
    \end{scope}
    
    \begin{scope}[shift={(0,0)}]
      \coordinate (0) at (-0.5,{sqrt(3)/2});
      \coordinate (1) at (0.5,{sqrt(3)/2});
      \coordinate (2) at (0,{sqrt(3)});
      \coordinate (3) at (-1,0);
      \coordinate (4) at (0,0);
      \coordinate (5) at (1,0);
      \coordinate (6) at (1.5,{sqrt(3)/2});
      \coordinate (7) at (1,{sqrt(3)});
      \coordinate (8) at (0.5,{sqrt(3)*3/2});
      \coordinate (9) at (-0.5,{sqrt(3)*3/2});
      \coordinate (10) at (-1,{sqrt(3)});
      \coordinate (11) at (-1.5,{sqrt(3)/2});
      \filldraw[greenpoly] (4)--(7)--(10)--cycle;
      \filldraw[bluepoly] (0)--(1)--(2)--cycle;
      \node () at (0,1.2) {$\textcolor{blue}{\rho}$};
      \draw[greenline] (3)--(8) (4)--(7) (5)--(9) (4)--(10) (6)--(11) (7)--(10);      
      \draw[blueline] (0)--(1)--(2)--cycle;
      \foreach \n in {0,...,2} {\draw (\n) node [bluedot] {};}
      \foreach \n in {3,...,11} {\draw (\n) node [reddot] {};}
    \end{scope}
    
    \begin{scope}[shift={(4,0)}]
      \coordinate (0) at (-0.5,{sqrt(3)/2});
      \coordinate (1) at (0.5,{sqrt(3)/2});
      \coordinate (2) at (0,{sqrt(3)});
      \filldraw[bluepoly] (0)--(1)--(2)--cycle;
      \node () at (0,1.2) {$\textcolor{blue}{\rho}$};
      \draw[blueline] (0)--(1)--(2)--cycle;      
      \foreach \n in {0,...,2} {\draw (\n) node [bluedot] {};}
    \end{scope}
  \end{tikzpicture}
  \caption{The natural filtration of the star of a triangle $\rho$ in
    the standard triangular tiling of the plane.\label{fig:filter}}
\end{figure}

\begin{lem}[Local deformations]\label{lem:local-deform}
  Let $\rho$ be a simplex in a flag complex $X$ and let $S = \cup_k
  S^{(k)}$ be the natural filtration of $S=\st(\rho)$.  For every $k\geq
  0$ there is a deformation retraction from $S^{(k)}$ to $S^{(k-1)}$.
\end{lem}

\begin{proof}
  Let $L^{(k)}$ be the $k$-skeleton of $\link(\rho)$.  For each
  $k$-simplex $\tau$ in $L^{(k)}$ with corresponding maximum
  connecting simplex $\sigma$ as described by
  Lemma~\ref{lem:max-connect}, we associate a deformation retraction
  from $\sigma$ onto the subcomplex $\sigma\setminus\tau$ as in
  Proposition~\ref{prop:deform}.  In fact, we claim that the
  deformation retractions associated to each $k$-simplex in $L^{(k)}$
  are compatible and can be performed simultaneously.  To see this,
  note that the only points that move under the deformation from
  $\sigma$ to $\sigma\setminus \tau$ are those contained in simplices
  that contain all of the vertices of $\tau$ and because $\tau$ is a
  maximal simplex in $L^{(k)}$, the only simplices in
  $S^{(k)}$ that contain $\tau$ are contained in $\sigma$ by
  Lemma~\ref{lem:max-connect}. Hence, the points moved
  by different deformations are pairwise disjoint.
\end{proof}

In Figure~\ref{fig:filter} it is easy to visualize the deformation
retractions between one step in the natural filtration of the star of
$\rho$ and the next.

\begin{prop}[Stars contract]\label{prop:stars-contract}
  If $\rho$ is a simplex in a flag complex $X$, then $\st(\rho)$ is a
  contractible subcomplex.
\end{prop}

\begin{proof}
  Let $\st(\rho) = S = \cup_k S^{(k)}$ be the natural filtration of
  the star of $\rho$.  By iteratively applying
  Lemma~\ref{lem:local-deform}, we see that each subcomplex $S^{(k)}$
  deformation retracts to the simplex $S^{(-1)} = \rho$ which is
  itself contractible.  Since each $S^{(k)}$ is contractible, the star
  of $\rho$ is a nested union of contractible simplicial complexes and
  thus contractible.
\end{proof} 

Easy counterexamples show that the flag complex requirement is
crucial.

\begin{example}[Noncontracting stars]\label{ex:noncontracting}
  If $X$ is the boundary of a simplex with at least $3$ vertices and
  $\rho$ is any proper subsimplex containing at least $2$
  vertices, then the star of $\rho$ is all of $X$, which is not
  contractible since it is homeomorphic to a sphere.
\end{example}

\section{Noncrossing Hypertrees}\label{sec:ncht}

In this section we apply Proposition~\ref{prop:stars-contract} to the
complex of noncrossing hypertrees.  We follow \cite{mccammond-ncht}
and call two edges, or more generally two polygons, \emph{weakly
  noncrossing} when they are disjoint or only intersect at a single
common vertex.  We begin with the notion of a noncrossing tree.

\begin{defn}[Noncrossing trees]\label{def:nc-trees}
  A \emph{tree} is a connected simplicial graph $T$ with no nontrivial
  cycles.  A \emph{noncrossing tree} is a tree $T$ whose vertices are
  the vertices of a convex polygon $P$, each edge is the convex hull
  of its endpoints and distinct edges of $T$ are weakly noncrossing.
\end{defn}

Noncrossing trees are examples of noncrossing hypertrees.

\begin{defn}[Noncrossing hypertrees]\label{def:ncht}
  A \emph{hypergraph} consists of a set $V$ of vertices and a
  collection of subsets of $V$ of size at least $2$ whose elements are
  called \emph{hyperedges}.  A \emph{hypertree}, roughly speaking, is
  a connected hypergraph with no hypercycles.  See
  \cite{mccammond-ncht} for a precise definition.  A \emph{noncrossing
    hypertree} is a hypertree whose vertices are those of a convex
  polygon $P$, each hyperedge is drawn as the convex hull of the
  vertices it contains, and distinct hyperedges are weakly
  noncrossing.  The set of all noncrossing hypertrees on $P$ is a
  poset under refinement.  In particular, one hypertree is below
  another if every hyperedge of the first is a subset of some
  hyperedge of the second.  In this ordering the noncrossing trees are
  its minimal elements and the noncrossing hypertree with only one
  hyperedge containing all of the vertices is its unique maximum
  element.
\end{defn}

The lefthand side of Figure~\ref{fig:ncht-poly} is a noncrossing
hypertree.  Noncrossing hypertrees are in bijection with certain types
of polygon dissections.

\begin{defn}[Polygon dissections]\label{def:poly}
  Let $P$ be a convex polygon in the plane.  A \emph{diagonal of $P$}
  is the convex hull of two vertices of $P$ that are not adjacent in
  the boundary cycle of the polygon.  Two diagonals are
  \emph{noncrossing} when they are weakly noncrossing as edges and a
  collection of pairwise noncrossing diagonals is called a
  \emph{dissecton of $P$}.  A dissection of an even-sided polygon $P$
  is itself \emph{even-sided} if the diagonals partition $P$ into
  smaller even-sided polygons.  An example is shown on the righthand
  side of Figure~\ref{fig:ncht-poly}.  The collection of all
  even-sided dissections of an even-sided polygon $P$ can be turned
  into a poset by declaring that one dissection is below another if
  the set of diagonals defining the first is a subset of the set of
  diagonals defining the second.  Since ideals in this poset are
  boolean lattices, it corresponds to a simplicial complex.  See
  \cite{mccammond-ncht} for further details.
\end{defn}

There is a natural bijection between noncrossing hypertrees in an
$n$-gon and the even-sided dissections of a $2n$-gon.

\begin{thm}[Noncrossing hypertrees and polygon dissections]\label{thm:ncht-poly}
  There is a natural order-reversing bijection between the poset of
  noncrossing hypertrees on a fixed number of vertices and the poset
  of dissections of an even-sided polygon with twice as many vertices
  into even-sided subpolygons through the addition of pairwise
  noncrossing diagonals.
\end{thm}

\begin{proof}
   \cite[Theorem~3.4]{mccammond-ncht}.
\end{proof}

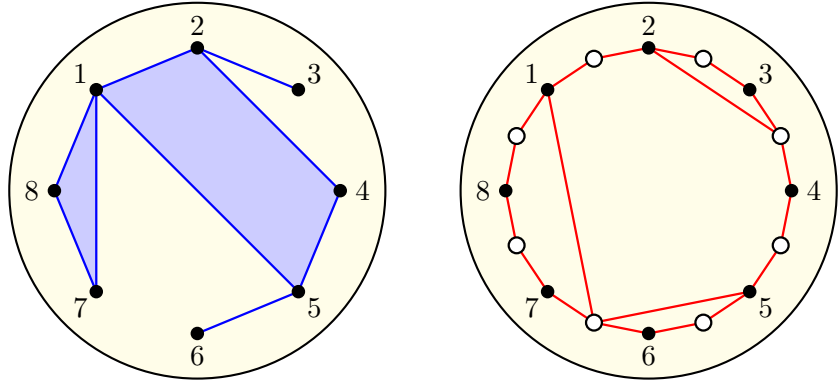
\begin{figure}
  \begin{tikzpicture}
    \begin{scope}[shift={(-3,0)},BluePoly]
      \node () [disk,minimum size=5cm] {};
      \foreach \n in {1,...,8} {
	\coordinate (\n) at (180+\n*-45:1.9cm) {};
	\node[black] () at (180+\n*-45:2.2cm) {\n};
      }
      \filldraw (1)--(2)--(4)--(5)--cycle;
      \filldraw (1)--(8)--(7)--cycle;
      \draw (2)--(3) (5)--(6);
      \foreach \n in {1,...,8} {\draw (\n) node [dot] {};}
    \end{scope}
    
    \begin{scope}[shift={(3,0)},BluePoly]
      \node () [disk,minimum size=5cm] {};
      \foreach \n in {1,...,8} {
	\coordinate (\n) at (180+\n*-45:1.9cm) {};
	\node[black] () at (180+\n*-45:2.2cm) {\n};
      }
      \foreach \n in {9,...,16} {\coordinate (\n) at (180-22.5+\n*-45:1.9cm) {};}
      \foreach \n in {1,...,16} {\draw[RedPoly] (\n*22.5:1.9cm)--(\n*22.5+22.5:1.9cm);}
      \draw[RedPoly] (11)--(2);
      \draw[RedPoly] (14)--(5);
      \draw[RedPoly] (14)--(1);
      \foreach \n in {1,...,8} {\draw (\n) node [dot] {};}
      \foreach \n in {9,...,16} {\draw[color=black,fill=white] (\n) circle (3pt);}
    \end{scope}
  \end{tikzpicture}
  \caption{A noncrossing hypertree with $8$ vertices and the
    corresponding even-sided dissection of a
    $16$-gon.\label{fig:ncht-poly}}
\end{figure}

\begin{example}[Hypertrees and dissections]\label{ex:ncht-poly} 
  The rough idea behind Theorem~\ref{thm:ncht-poly} is easy to
  describe.  Given an even-sided subdivision of a $2n$-gon such as the
  one shown on the righthand side of Figure~\ref{fig:ncht-poly}, we
  alternately color the vertices black and white and then create a
  hypertree using the convex hulls of the black vertices in each
  even-sided subpolygon.  In the other direction we add a white dot in
  between every pair of vertices of the $n$-gon and then includes a
  diagonal corresponding to each pair of hyperedges that share a
  vertex and are adjacent in the linear ordering of hyperedges that
  share that vertex.  The other end of the diagonal is the unique
  white dot to which it can be connected without crossing a hyperedge.
\end{example}

\begin{defn}[Noncrossing hypertree complex]\label{def:ncht-complex}
  Let $P$ be a convex polygon.  The identification given in
  Theorem~\ref{thm:ncht-poly} can be used to define a simplicial
  complex $X$ called the \emph{noncrossing hypertree complex of $P$}.
  Its simplices are the noncrossing hypertrees on $P$ with more than
  one hyperedge.  The vertices of $X$ are labeled by the noncrossing
  hypertrees with exactly two hyperedges and these correspond to
  polygonal dissections with only one diagonal.  At the other extreme,
  the maximal simplices of $X$ are labeled by noncrossing trees.
\end{defn}

From the polygon dissection viewpoint, the following result is
immediate.

\begin{lem}[Flag]\label{lem:flag}
  The noncrossing hypertree complex is a flag complex.
\end{lem}

\begin{proof}
  The vertices correspond to subdivisions containing a single diagonal
  and two vertices are connected by an edge iff the corresponding
  diagonals are noncrossing.  In particular, any complete graph in the
  $1$-skeleton has vertices labeled by dissections with diagonals that
  can be overlaid to obtain a new polygon dissection which contains
  all of them and this dissection corresponds to the noncrossing
  hypertree that labels the simplex with this complete graph as its
  $1$-skeleton.
\end{proof}

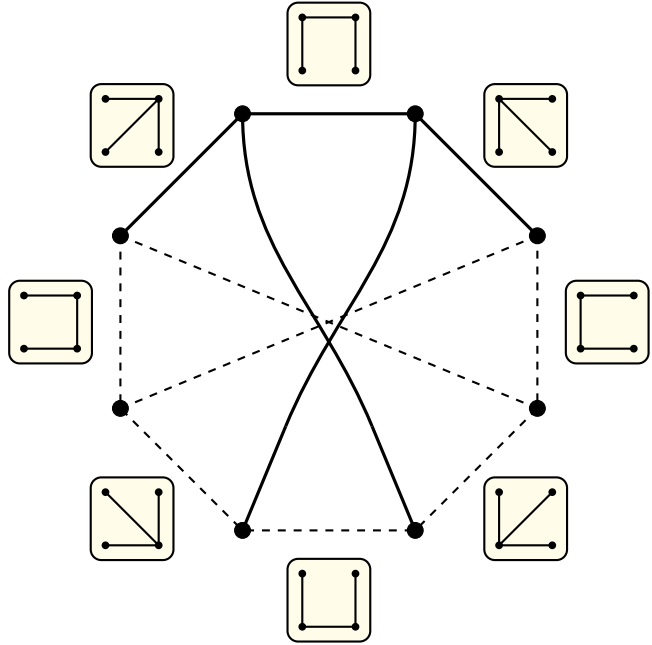
\begin{figure}
  \begin{tikzpicture}
    \begin{scope}[shift={(3,0)}]
      \foreach \n in {1,...,8} {\coordinate (\n) at (22.5+\n*45:3cm) {};}
      \foreach \n in {9,...,12} {\coordinate (\n) at (157.5 + \n*45:1.5cm) {};}
      \draw[very thick] (8)--(1)--(2)--(3);
      \draw[thick,dashed] (3)--(4)--(5)--(6)--(7)--(8);
      \draw[very thick] (1) to [out=-90,in=67.5] (10);
      \draw[very thick] (2) to [out=-90,in=112.5] (11);
      \draw[very thick] (10) to (5);
      \draw[very thick] (11) to (6);
      \draw[thick,dashed] (3) to (7);
      \draw[thick,dashed] (4) to (8);
      \foreach \n in {1,...,8} {\coordinate (x\n) at (45*\n:3.7cm) {};\node[plate] () at (x\n) {};}
      \foreach \n in {1,...,8} {\draw[fill=black] (\n) circle (3pt);}
      \begin{scope}[thick]
	\begin{scope}[shift={(x1)}] \makepoints \draw (2)--(1)--(4) (1)--(3); \drawpoints \end{scope}
	\begin{scope}[shift={(x2)}] \makepoints \draw (4)--(1)--(2)--(3); \drawpoints \end{scope}
	\begin{scope}[shift={(x3)}] \makepoints \draw (1)--(2)--(3) (2)--(4); \drawpoints \end{scope}
	\begin{scope}[shift={(x4)}] \makepoints \draw (1)--(2)--(3)--(4); \drawpoints \end{scope}
	\begin{scope}[shift={(x5)}] \makepoints \draw (4)--(3)--(2) (1)--(3); \drawpoints \end{scope}
	\begin{scope}[shift={(x6)}] \makepoints \draw (1)--(4)--(3)--(2); \drawpoints \end{scope}
	\begin{scope}[shift={(x7)}] \makepoints \draw (1)--(4)--(3) (2)--(4); \drawpoints \end{scope}
	\begin{scope}[shift={(x8)}] \makepoints \draw (2)--(1)--(4)--(3); \drawpoints \end{scope}
      \end{scope}
    \end{scope}
  \end{tikzpicture}
  \caption{The noncrossing hypertree complex $X$ of a square $P$.  The
    edges of $X$ are dashed when the tree $T$ labeling the edge
    contains the boundary edge $e$ along the bottom of the square $P$
    and solid when $T$ does not contain $e$. The top edge of the
    figure is the edge labeled by the tree $T_e$.\label{fig:ncht4}}
\end{figure}

Let $P$ be a convex polygon.  We call an edge $e$ connecting two
vertices of $P$ a \emph{boundary edge} if it is contained in
the boundary of $P$.

\begin{example}[Unused boundary edge]
  Let $P$ be a square with horizontal and vertical sides, let $e$ be
  the bottom edge of $P$ and let $X$ be the noncrossing hypertree
  complex of $P$.  In this case the complex $X$ is the nonplanar graph
  with $8$ vertices and $12$ edges shown in Figure~\ref{fig:ncht4}.
  The edges of $X$ are labeled by noncrossing trees and the labels for
  $8$ of these edges are shown.  The remaining $4$ edges in $X$ are
  labeled by noncrossing trees that look like an $N$ or a $Z$ up to
  rotation and/or reflection.  An edge has been drawn as a dashed line
  when its label contains the boundary edge $e$ and it has been drawn
  as a solid line when its label does not contain $e$.  Let $T_e$ be
  the tree formed by the $3$ sides of $P$ other than $e$;  $T_e$ is the
  label for the edge along the top of the graph $X$, and the star of
  this edge is depicted as the subgraph formed by the solid edges.
\end{example}

The noncrossing trees that do not contain a particular boundary edge
$e$ always form a subcomplex that is the star of a simplex.

\begin{lem}[Unused boundary edge]\label{lem:star}
  Let $P$ be a convex polygon and let $X$ be the noncrossing hypertree
  complex of $P$.  If $e$ is a boundary edge of $P$ and $T_e$ is the
  noncrossing tree consisting of all boundary edges other than $e$,
  then the star of the simplex labeled by $T_e$ in $X$ contains the
  simplex labeled by a noncrossing tree $T$ if and only if $T$ does
  not contains the edge $e$.
\end{lem}

\begin{proof}
  Let $T$ is a noncrossing tree in $P$.  First note that $T$ does not
  contain the boundary edge $e$ if and only if the 
  corresponding polygon
  dissection contains a diagonal with one endpoint at the white dot
  between the endpoints of $e$.  Moreover, the set of all of the
  diagonals ending at this white dot are pairwise noncrossing and form
  the polygon dissection that corresponds to the tree $T_e$.  And
  since single diagonals label the vertices of the noncrossing
  hypertree complex, this means that $T$ does not contain the boundary
  edge $e$ if and only if the simplex labeled $T$ contains a vertex of the
  simplex labeled $T_e$, which is equivalent to the simplex labeled $T$ 
  being in the star of the simplex labeled $T_e$.
\end{proof}

We are now ready to apply Proposition~\ref{prop:stars-contract}.

\begin{thm}[Unused boundary edge]\label{thm:ncht-edge}
  The simplices of the noncrossing hypertree complex labeled by
  noncrossing trees that omit a fixed boundary edge $e$ form a
  contractible subcomplex.
\end{thm}

\begin{proof}
  By Lemma~\ref{lem:flag} the noncrossing hypertree complex is a flag
  complex and by Lemma~\ref{lem:star} the subcomplex under discussion
  is the star of a simplex.  Proposition~\ref{prop:stars-contract}
  completes the proof.
\end{proof}

\section{Noncrossing Partitions}\label{sec:ncpart}

In this section we review the connection between noncrossing
partitions to noncrossing hypertrees and prove
Theorem~\ref{main:edge}.

\begin{defn}[Noncrossing partitions]\label{def:ncpart}
  Let $P$ be a convex $n$-gon.  A \emph{noncrossing partition of $P$}
  is a partition of its vertex set so that the convex hulls of the
  blocks of the partition are pairwise disjoint.  The noncrossing
  partitions of $P$ are ordered by refinement so that one partition is
  below another if every block of the first is a subset of some block
  of the second.  The result is a bounded graded lattice called the
  \emph{noncrossing partition lattice}.  Since the combinatorics only
  depend on the integer $n$ and not on the choice of polygon $P$, the
  noncrossing partition lattice of $P$ is denoted $\ncp_n$.  The
  \emph{noncrossing partition link} is the order complex of the
  noncrossing partition lattice with its bounding elements removed.
  This means that the maximal simplices in the noncrossing partition
  link are labeled by the maximal chains in the noncrossing partition
  lattice.
\end{defn}

\begin{defn}[Noncrossing permutations]\label{def:ncperm}
  Noncrossing partitions in a convex polygon $P$ can be reinterpreted
  as permutations of the vertices of $P$.  To do this for a particular
  partition $\pi$ simply send each vertex $v$ of $P$ to the next
  vertex in the clockwise order of the vertices in the boundary cycle
  of the convex hull of the block of $\pi$ containing $v$.  The result
  is called a \emph{noncrossing permutation}.  Next we assign a
  (noncrossing) permutation to each pair of comparable elements in the
  noncrossing partition lattice by multiplying the permutation
  associated to the bigger element by the inverse of the permutation
  associated to the smaller element.  Under this procedure the
  permutations assigned to the covering relations are transpositions.
  And the collection of the transposition labels on the adjacent
  elements in a maximal chain correspond to the edges of a noncrossing
  hypertree in $P$ \cite{mccammond-ncht}.
\end{defn}

\begin{defn}[Properly ordered noncrossing trees]\label{def:proper-order}
  Let $T$ be a noncrossing tree in a polygon $P$.  When $T$ comes
  equipped with a linear ordering of its edges, we call $T$ an
  \emph{ordered noncrossing tree}.  The ordering is a \emph{proper
    ordering} when it extends the local linear orderings of the edges
  that share any particular vertex \cite{mccammond-ncht}.  This is
  equivalent to saying that the product of the transpositions
  corresponding to the edges multiplied in this order produces the
  single cycle that corresponds to the boundary cycle of $P$.
\end{defn}

\begin{lem}[Factorizations and trees]\label{lem:factor-tree}
  There are natural bijections between the maximal chains in the
  noncrossing partition lattice $\ncp_{n+1}$, factorizations of the
  $(n+1)$-cycle $(1,2,\ldots,n+1)$ into $n$ transpositions, and
  properly ordered noncrossing trees with $n+1$ vertices and $n$
  edges.
\end{lem}

\begin{proof}
  \cite[Theorem~7.8, Corollary~7.9]{mccammond-ncht}.
\end{proof}

Since the maximal simplices of the noncrossing hypertree complex are
labeled by noncrossing trees and the maximal simplices of the
noncrossing partition link are labeled by \emph{properly ordered} 
noncrossing trees, it is not too surprising that there is a close relationship
between the two simplicial complexes.  In fact, they are homeomorphic
to each other, and this is one of the main results proved in
\cite{mccammond-ncht}.

\begin{thm}[Hypertrees and partitions]\label{thm:ncht-ncp}
  For any convex polygon $P$ there is a natural homeomorphism between
  the noncrossing partition link on $P$ and the noncrossing hypertree
  complex on $P$.  Moreover, the maximal simplices of the former
  labeled by the various proper orderings of a fixed noncrossing tree
  form a subcomplex that is identified with a single maximal simplex
  of the latter labeled by the common underlying noncrossing
  hypertree.
\end{thm}

\begin{proof}
  See \cite[Section~$8$]{mccammond-ncht}.
\end{proof}

Theorem~\ref{thm:ncht-ncp} allows us to identify subcomplexes of the
noncrossing hypertree complex with their images in the noncrossing
partition link.  In particular, Theorem~\ref{main:edge} is an
immediate consequence of Theorem~\ref{thm:ncht-edge} and
Theorem~\ref{thm:ncht-ncp}.  We conclude this section with a concrete
example to show what Theorem~\ref{thm:ncht-ncp} looks like when the
polygon $P$ is a square.

\begin{figure}
  \begin{tikzpicture}
    \begin{scope}[shift={(3,0)}]
      \foreach \n in {1,...,8}{\coordinate (\n) at (22.5+\n*45:3cm) {};}
      \foreach \n in {9,...,12}{\coordinate (\n) at (157.5+\n*45:1.5cm) {};}
      \draw[very thick] (8)--(1)--(2)--(3);
      \draw[thick,dashed] (3)--(4)--(5)--(6)--(7)--(8);
      \draw[very thick] (1) to [out=-90,in=67.5] (10);
      \draw[very thick] (2) to [out=-90,in=112.5] (11);
      \draw[very thick] (10) to (5) (11) to (6);
      \draw[thick,dashed] (3) to (7) (4) to (8);
      \foreach \n in {1,...,12}{\draw[fill=black] (\n) circle (3pt);}
    \end{scope}
  \end{tikzpicture}
  \caption{The noncrossing partition link of a square $P$. The edges
    of the link are dashed when the properly ordered noncrossing tree
    $T$ labeling the edge contains the boundary edge $e$ along the
    bottom of the square $P$ and solid when $T$ does not contain $e$.
    \label{fig:ncpl4}}
\end{figure}
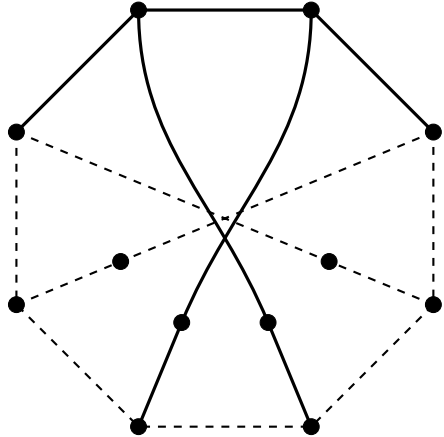

\begin{example}[Noncrossing partition link of a square]\label{ex:ncpl4}
  Let $P$ be a square with horizontal and vertical sides and let $e$
  be the bottom edge of $P$.  The noncrossing partition link of $P$ is
  a nonplanar graph with $12$ vertices and $16$ edges.  See
  Figure~\ref{fig:ncpl4}.  The edges of the complex are labeled by
  properly ordered noncrossing trees.  Most noncrossing trees on $4$
  vertices have a unique proper ordering so the outside portion of the
  figure looks the same as in Figure~\ref{fig:ncht4}.  The noncrossing
  trees that look like an $N$ or a $Z$ up to rotation and/or
  reflection have two proper orderings.  As a consequence the single
  edges crossing through the middle of the image in
  Figure~\ref{fig:ncht4} become two edges in Figure~\ref{fig:ncpl4}.
  An edge has been drawn as a dashed line when its label contains the
  boundary edge $e$ and it has been drawn as a solid line when its
  label does not contain $e$.  Let $T_e$ be the tree formed by the $3$
  sides of $P$ other than $e$.  It has only one proper ordering and is
  the label for the edge along the top of the graph $X$, but note that the
  subgraph formed by the solid edges is no longer the star of the edge
  labeled $T_e$.  It is, however, homeomorphic to the star of the edge
  labeled $T_e$ in the noncrossing hypertree complex.
\end{example}

\section{Parking Functions}\label{sec:park}

In this section we recall Stanley's elegant bijection between parking
functions and maximal chains in the noncrossing partition lattice and
we use this to prove explicit versions of Theorems~\ref{main:last}
and~\ref{main:space}.

\begin{defn}[Parking functions and undesired spaces]\label{def:parking}
  Recall that an $n$-tuple $(a_1,\ldots, a_n)$ of positive integers is
  called a \emph{parking function} if, once its entries have been
  sorted into weakly increasing order, the $i$-th entry is at most
  $i$. The set of all parking functions of length $n$ is denoted
  $\pkfn_n$.  In \cite{reu16} the first author and his coauthors
  introduced a filtration of parking functions by undesired spaces.
  An integer $k$ is called an \emph{undesired parking space} for a
  specific parking function $(a_1, \ldots, a_n) \in \pkfn_n$ if $k$
  does not appear in this $n$-tuple.  The language refers to the model
  where the $i$-th car wants to park in the $a_i$-th parking space on
  a one-way street.  Let $\pkfn_{n,k}$ be the set of parking functions
  of length $n$ with the property that $k$ is the largest undesired
  parking space.  Notice that every parking function is either a
  permutation of the numbers $1$ through $n$ or it belongs to
  $\pkfn_{n,k}$ for some unique positive integer $k$.
\end{defn}

\begin{defn}[Stanley's bijection]\label{def:stanley}
  Let $P$ be a convex $(n+1)$-gon with vertex labels $1, \ldots,n+1$
  in clockwise-increasing order and note that the noncrossing
  permutation associated to trivial partition with only one block is
  the $(n+1)$-cycle $(1,2,\ldots,n+1)$.  For each factorization of $c$
  into $n$ transpositions, we can read off the smallest number in each
  transposition to produce an $n$-tuple with entries in the set
  $\{1,2,\ldots,n\}$.  Although it is not immediately obvious, Richard
  Stanley proved that this procedure establishes a natural bijection
  between the maximal chains in $\ncp_{n+1}$ thought of as
  factorizations of $c$ into $n$ transpositions, and the set $\pkfn_n$
  of length $n$ parking functions \cite{stanley97}.  By
  Lemma~\ref{lem:factor-tree} the set $\pkfn_n$ is also in bijection
  with the set of properly ordered noncrossing trees on $n+1$
  vertices.
\end{defn}

The parking functions where the last space is undesired can be
characterized as those factorizations that avoid a particular boundary
transposition.

\begin{lem}[Trees and parking functions]\label{lem:tree-park}
  Let $P$ be a convex $(n+1)$-gon with vertex labels $1,\ldots,n+1$
  in clockwise-increasing order and let $e$ be the boundary edge
  connecting the vertices labeled $n$ and $n+1$.  The properly ordered
  noncrossing trees that omit $e$ correspond to parking functions in
  $\pkfn_{n,n}$, i.e. the parking functions where no car wants to park
  in the last parking space.
\end{lem}

\begin{proof}
  Under Stanley's bijection (Definition~\ref{def:stanley}), it is
  clear that $n$ occurs in the parking function if and only if 
  the transposition $(n,n+1)$ occurs in the factorization of the 
  $(n+1)$-cycle $c =(1,2,\ldots,n+1)$ into $n$ transpositions, 
  and this is true if and only if the
  corresponding properly ordered noncrossing tree contains the
  boundary edge $e$ (Lemma~\ref{lem:factor-tree}).
\end{proof}

If we consider the Hasse diagram of a finite bounded poset as a directed
graph whose vertices are its elements and with a directed edge for
each covering relation, then the maximal chains in the poset
correspond to directed paths from its unique minimum element to its
unique maximum element.

\begin{defn}[Parking functions, posets and subcomplexes]
  For each subset $A \subset \pkfn_n$ we define a poset $\pos(A)$ as
  the poset whose Hasse diagram is the union of the directed paths
  that correspond to the parking functions in $A$ under Stanley's
  bijection.  For any bounded poset $P$ we write $\link(P)$ for the
  order complex of $P$ with its bounding elements removed.  With this
  definition, the complex $\link(\pos(A))$ is the subcomplex of the
  noncrossing partition link $\link(\ncp_{n+1})$ formed by the
  simplices labeled by parking functions in $A$.
\end{defn}

The following result is an explicit version of
Theorem~\ref{main:last}.

\begin{thm}[Undesired last parking space]\label{thm:last}
  The simplices labeled by the parking functions in $\pkfn_{n,n}$ form
  a subcomplex $\link(\pos(\pkfn_{n,n}))$ in the noncrossing partition
  link $\link(\ncp_{n+1})$ that is contractible.
\end{thm}

\begin{proof}
  This is an immediate consequence of Theorem~\ref{main:edge} and
  Lemma~\ref{lem:tree-park}.
\end{proof}

In order to extend this result to the other collections $\pkfn_{n,k}$
we need to quote a decomposition result from \cite{reu16}.

\begin{lem}[Decomposition]\label{lem:decomp}
  For all positive integers $k + \ell = n$ with $k>1$, the poset
  $\pos(\pkfn_{n,k})$ is isomorphic to the product $\pos(\pkfn_{k,k})
  \times \bool_\ell$, where $\bool_\ell$ is the Boolean lattice of rank
  $\ell$.
\end{lem}

\begin{proof}
 \cite[Theorem~3.5]{reu16}.
\end{proof}

When a poset splits as a direct product of two smaller posets, the
link of the whole can be constructed from the link of each factor.

\begin{lem}[Links and products]\label{lem:link-prod}
  If $P= P_1 \times P_2$ is a bounded graded poset that splits as a
  product of smaller bounded graded posets, then $\link(P)$ is
  homeomoprhic to $\link(P_1) \ast \link(P_2)$, the spherical join of
  the smaller links.  As a consequence, when $\link(P_1)$ is
  contractible, so is $\link(P)$.
\end{lem}

\begin{proof}
  The first assertion is an easy exercise.  For the second assertion
  let $L_i = \link(P_i)$.  Since $L_1$ is contractible, it is homotopy
  equivalent to a point.  Thus $L_1 \ast L_2$ is homotopy equivalent
  to the spherical join of a point and $L_2$, which is homeomorphic to
  the cone on $L_2$, which is contractible.
\end{proof}

Using Lemma~\ref{lem:decomp} and Lemma~\ref{lem:link-prod} we prove
the following explicit version of Theorem~\ref{main:space}.

\begin{thm}[Undesired parking space]\label{thm:space}
  For all integers $1 < k \leq n$, the simplices labeled by the
  parking functions in $\pkfn_{n,k}$ form a contractible subcomplex of
  the noncrossing partition link $\link(\ncp_{n+1})$.
\end{thm}

\begin{proof}
  For $k=n$, this follows from Theorem~\ref{thm:last}.  For $1 < k <
  n$, let $\ell = n-k > 0$. By Lemma~\ref{lem:decomp} the poset
  $\pos(\pkfn_{n,k})$ splits as a product of $\pos(\pkfn_{k,k})$ and
  the Boolean lattice $\bool_\ell$ and since
  $\link(\pos(\pkfn_{k,k}))$ is contractible by
  Theorem~\ref{thm:last}, $\link(\pos(\pkfn_{n,k}))$ is contractible
  by Lemma~\ref{lem:link-prod}.
\end{proof}

\newcommand{\etalchar}[1]{$^{#1}$}

\end{document}